\documentclass[graybox]{svmult}

\usepackage{makeidx}         
\usepackage{graphicx}        

\usepackage{amsmath,bm,amsfonts,graphicx,tikz,pgfplots, pdfsync}
\usetikzlibrary{patterns}

\usepackage{url}

\usepackage{algorithmic}
\usepackage{amsopn}
\usepackage{tikz,pgfplots,bm}
\usetikzlibrary{patterns}

\usepackage{ulem}
\usepackage{xcolor}

\usepackage{amsopn}

\def \e {{\rm e}}
\def \R{{\mathbb{R}}}
\def \S{{\mathbb{S}}}

\newcommand{\ds}{\displaystyle}
\newcommand{\be}{\begin{equation}}
\newcommand{\ee}{\end{equation}}

\def\nablav{\bm\nabla}
\def\d{\mathrm{d}}

\def\hh {{h}}

\def\vx {{\bm x}}
\def\vy {{\bm y}}

\def\omegav {{\bm\omega}}
\def\vn {{\bm n}}
\def\vK {{\bm K}}

\def\vg {{\bm g}}

\def\vom{{\bm\omega}}
\def\SS{{\mathbb{S}^2}}

\begin{document}


\title*{Simulation of the 3D Radiative Transfer}
\subtitle{with Anisotropic Scattering for Convective Trails}
\author{Olivier Pironneau
 \and  Pierre-Henri Tournier}
\institute{Olivier Pironneau
\at Applied Mathematics, Jacques-Louis Lions Lab, Sorbonne
Universit\'e, 75252 Paris cedex 5, France,
\email{olivier.pironneau@sorbonne-universite.fr} 
\and
Pierre-Henri Tournier
\at Applied Mathematics, Jacques-Louis Lions Lab, Sorbonne
Universit\'e, 75252 Paris cedex 5, France,
\email{Pierre-Henri.Tournier@sorbonne-universite.fr} 
}

\maketitle

\section{Introduction}

Airplane combustion engines generate \texttt{CO}$_2$ (carbon dioxide), a gas which stays in the Earth atmosphere for $\sim$ 500 years before being partially absorbed by the oceans.

Currently used fuels and potential future substitutes like hydrogen, synthetic or agricultural fuel combustion generate 	also  \texttt{NO}$_x$ (Nitrous oxides). Typically a nitrous oxide stays in the atmosphere $\sim$ 100 years.

Carbon dioxide, nitrous oxides, methane, ozone and water vapor are ``green house gases'', (GHG), in the sense that atmospheric  and past climates measurements indicate that they are most likely responsible for the additional heat received by Earth (global warming).

Air travel -- around 100.000 flights/day -- produces 3\% of \texttt{CO}$_2$ (or  12\% of what is produced by all transport vehicles). 
It is still small, but the contribution of airplanes to global warming is expected to rise from today's $0.024W/m^2$ to $0.084W/m^2$ by 2050 \cite{Sausen}.
\\\\
A  cloud similar to a cirrus, called ``contrail (short for condensation trail) may appear in the airplane wake if atmospheric pressure and temperature are on the left side of a threshold  curve, and  further left of that pressure versus temperature curve the contrail will persist \cite{roberto}.  Several studies claim that these clouds have a warming effect perhaps 3 times stronger than the one caused by the combustion gases \cite{lee}. 
\\\\
In the past few years the authors have worked to see if these claims could be validated by a numerical simulation of the fundamental equations of physics for these problems \cite{CBOP}, \cite{DIA}, \cite{FGOP}, \cite{FGOP3}.  we approach the problem from an applied mathematics point of view, with convergence error estimation and computational efficiency in mind.
\\\\
The formation of a contrail has been simulated by solving the Navier-Stokes equations (NSE) with chemistry for the engine exhaust and phase change for the ice formation in the airplane wake (see \cite{roberto}); it is the right approach to understand the mechanisms of the formation of contrails.  Once the cloud is sufficiently developed one may study its effect on the absorption and scattering coefficients of the 
Radiative Transfer Equations (RTE) in the Earth atmosphere (see the data in \cite{libtran}) and then solve the RTE-NSE system.  The numerical simulations of RTE in one dimension is the object of intense research \cite{I3RC}, \cite{libtran}, but mostly without coupling with a temperature equation or NSE.

In \cite{JCP} and \cite{JCP2} we have proposed a numerical algorithm to solve the RTE in 3D based on an integro-differential formulation and iterations using ${\mathcal H}$-matrices to speed-up the computation of the integrals.  In this article we show how the method can be extended to handle an important class of non-isotropic scattering for the atmosphere.

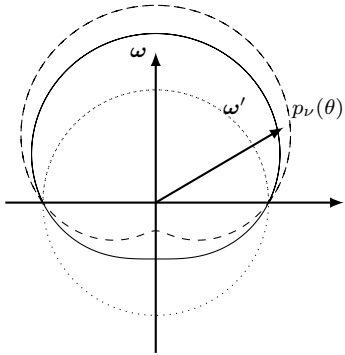
\begin{figure}[htbp]
\begin{center}
\begin{tikzpicture}
  \draw[thick,->,>=latex] (-2,0)--(2.5,0) node[above] {$~$};
   \draw[thick,->,>=latex] (0,0)--(1.7,1.) node[above] {$\vom'~~~~~~ p_\nu(\theta)$};
 \draw[thick,->,>=latex] (0,-2)--(0,2) node[left] {$\vom$};
  \draw[domain=0:540,scale=1.5,samples=500] plot (\x:{1+0.5*sin(\x)});
  \draw[dashed,domain=0:540,scale=1.5,samples=500] plot (\x:{1+0.75*sin(\x)});
  \draw[dotted,domain=0:540,scale=1.5,samples=500] plot (\x:{1+0.*sin(\x)});
\end{tikzpicture}
\caption{This polar plot shows the probability $p(\theta)$ that a photon in the direction $\vom$ scatters in the direction $\vom'$ : $p_\nu(\theta)= 1+ \tfrac12\vom\cdot\vom'=1+ \beta \cos(\theta)$ where $\theta$ is the angle $(\vom,\vom')$.  The solid curve is for $\beta=\tfrac12$, the dashed curve is for $\beta=0.75$ and the dotted circle is the isotropic case is $\beta=0$.}
\label{polar}
\end{center}
\end{figure}

It is known that the preferred direction of scattering in clouds is the initial direction.  It is modelled by an anisotropic probability of scattering (also called phase function) from $\vom$ to $\vom'$,  $p_\nu(\vom,\vom')$ (see Figure \ref{polar}). 

\section{The Radiative Transfer Equations}
When molecular viscosity and wind convection are ignored,  the temperature $T$ in a medium exposed to electromagnetic waves satisfies  the RTE  as explained in \cite{POM}.  It involves a frequency dependent radiation intensity field ${I}_\nu(\vx,\omegav)$  at  position $\vx$ in the physical domain $\Omega$ and in each direction $\omegav$. For all 
$ \{\vx,\omegav,\nu\}\in\Omega\times\S_2\times\R^+$,
\begin{equation}
\begin{aligned}
\label{onea}&
\omegav\cdot\nablav_\vx {I}_\nu+\rho\bar\kappa_\nu a_\nu\left[{I}_\nu-{\frac{1}{4\pi}\int_{\S^2}} p_\nu(\omegav,\omegav'){I}_\nu(\omegav'){\d}\omega'\right]
\\&
\hskip6cm = \rho\bar\kappa_\nu(1-a_\nu) [B_\nu(T)-{I}_\nu],
\\  &
 \int_0^\infty{\int_{\S^2}}\rho\bar\kappa_\nu(1-a_\nu) [B_\nu(T)-{I}_\nu] {\d}\omega {\d}\nu =0,
\end{aligned}
\end{equation}
where $\SS$ is the unit sphere, 
$\ds B_\nu(T)=\frac{2 \hbar \nu^3}{c^2[{\rm e}^\frac{\hbar\nu}{k T}-1]}$ is the Planck function,
  $\hbar,c,k$ are the Planck constant, the speed of light in the medium and  
the Boltzmann constant.
The absorption coefficient $\kappa_\nu:=\rho\bar\kappa_\nu$, where $\rho$ is the medium density,  comes from nuclear physics, but for our purpose it is seen as the percentage of radiation absorbed per unit length. The scattering albedo is  $a_\nu\in(0,1)$ and $p_\nu(\omegav,\omegav')$ is the probability that a ray in direction $\omegav$ scatters in direction $\omegav'$.  

With appropriate boundary conditions, existence of solution  has been established by \cite{POR} and \cite{DIA}. In the lat\textcolor{red}{t}er, convergence  of the following scheme was proved:
\begin{equation}
\begin{aligned}
\label{oneaa}&
\omegav\cdot\nablav I^{n+1}_\nu+\kappa_\nu I^{n+1}_\nu =
{\frac{\kappa_\nu}{4\pi}\int_{\S^2}} p_\nu(\omegav,\omegav')I^{n}_\nu(\omegav'){d}\omega'
+  \kappa_\nu(1-a_\nu) B_\nu(T^{n}),
\\  &
 \int_0^\infty{\int_{\S^2}} \kappa_\nu(1-a_\nu) [B_\nu(T^{n+1})-I^{n+1}_\nu] {d}\omega {d}\nu =0.
\end{aligned}
\end{equation}
By the maximum principle for the first equation and the monotony of $T\mapsto B_\nu(T)$, 
\[
I^n\prec I^{n-1},~ T^n\prec T^{n-1}\quad\implies\quad I^{n+1}\prec  I^n,\qquad \implies T^{n+1}\prec T^n,
\]
where $a\prec b$ means $a(\vx)\le b(\vx)$, for all $\vx$ but not $a(\vx)=b(\vx)$ everywhere.
Hence by choosing $T^0=0$ and $I^0=0$, the positivity of $T^1$ and $I^1$ implies that $T^1\prec T^{0}$ , $I^1\prec I^{0}$ so that , by induction, a strictly increasing sequence is obtained.  Similarly a decreasing sequence is obtained if we manage $T^1\succ T^{0}$, $I^1\succ I^{0}$.  As it was observed in \cite{FGOP3} that the increasing sequences converge faster we will focus on that one only.

\section{Integral Formulation}
Consider an anisotropic scattering density 
\[
p_\nu(\vom,\vom')=1+\beta_\nu\vom\cdot\vom'.
\]
Denote by $\Gamma$ the boundary of $\Omega$.  One must find the radiation  intensity $I_\nu(\vx,\vom)$ at all points  $\vx\in\Omega$, for all directions all $\vom\in\SS$ and all radiation frequencies $\nu\in\R_+$, by solving the radiative transfer equations (RTE):
\begin{align}
&\vom\cdot\nabla_x I_\nu+{\kappa_\nu} I_\nu={\kappa_\nu}(1-a_\nu)\left(B_\nu(T)+ \beta_\nu\vom\cdot  \vK_\nu\right)
+{\kappa_\nu}a_\nu J_\nu\,,
\label{RTa}
\\&
 J_\nu(\vx):=\tfrac1{4\pi}\int_{\SS}I_\nu\d\omega,\quad \vK_\nu(\vx):=\tfrac1{4\pi}\int_{\SS}\vom I_\nu(\vx,\vom)\d\omega\,,
\label{RTaa}
\\
&\int_0^\infty{\kappa_\nu}(1-a_\nu)(J_\nu-B_\nu(T))\d\nu=0\,,
\label{RTb}
 \\ &
I_\nu(\vx,\vom) = R_\nu(\vx,\vom) I_\nu(\vx,\vom-2(\vn\cdot\vom)\vn) + Q_\nu(\vx,\vom),
\cr & \hskip 3cm 
\hbox{ on }\Sigma:=\{(\vx,\vom)\in\Gamma\times\SS~:~\vom\cdot\vn(\vx)<0\,\}.
\label{RTc}
\end{align}
 In  \eqref{RTc}, $Q_\nu$ is the radiation source and $R_\nu$ is the portion of radiation which is reflected by the boundary; $\vn(\vx)$ is the outer normal of $\Gamma$ at $\vx$. $\kappa_\nu>0$ and $a_\nu\in[0,1]$ are the absorption and scattering coefficients; in general they depend on $\nu$ and $\vx$.

The general solution of \eqref{RTa} is
\[
I(\vx,\vom) = I(\vx_\Sigma(\vx,\vom),\vom)\e^{-\int_0^{\tau_{\vx,\vom}}\kappa(\vx-\vom s)\d s}
+
\int_0^{\tau_{\vx,\vom}} \e^{-\int_0^s\kappa(\vx-\vom s')\d s'} S_\nu(\vx-\vom s,\vom)\d s,
\]
where $\tau_{\vx,\vom}$ is the length $|\vx-\vx_{\textcolor{red}{\Sigma}}(\vx,\vom)|$, and ${\cal S}_\nu$ denotes its right-hand side,
\begin{equation}\label{ssource}
{\cal S}_\nu(\vx,\vom) = {\kappa_\nu(\vx)}(1-a_\nu(\vx))\left(B_\nu(T(\vx))+ \beta_\nu(\vx)\vom\cdot  \vK_\nu(\vx)\right)
+{\kappa_\nu(\vx)}a_\nu(\vx) J_\nu(\vx)\,.
\end{equation}
\subsection{Solution in Absence of Reflective Boundaries}
Let us denote $\int_{[\vx,\vy]}\kappa := \int_0^{|\vy-\vx|} \kappa(\vx+s(\vy-\vx))\d s$.  When $R_\nu=0$, the following holds:
\begin{proposition}\label{prop2}
\begin{align}&
J_\nu(\vx) = S^E_\nu(\vx) + {\mathcal J}[{\cal S}_\nu](\vx),
\\& 
S^E_{\nu}(\vx) =
 \tfrac{1}{4\pi}\int_\Gamma Q_\nu(\vy,\tfrac{\vy-\vx}{|\vy-\vx|})
 \frac{[(\vy-\vx)\cdot\vn(\vy)]_-}{|\vy-\vx|^3}\e^{-\int_{[\vx,\vy]}\kappa}\d\Gamma(\vy),
\label{SE1}
\\& 
{\mathcal J}[{\cal S_\nu}](\vx) = 
 \tfrac1{4\pi}\int_{\Omega} \frac{\e^{-\int_{[\vx,\vy]}\kappa}}{|\vy-\vx|^2}
  {\cal S}_\nu(\vy,\frac{\vy-\vx}{|\vy-\vx|})\d\vy
  \label{SE23}.
\end{align}
\end{proposition}

  Averaging \eqref{ssource} on $\SS$ after multiplication by $\vom$ leads to
\begin{equation}\label{gen1}
\begin{aligned}
\vK_\nu(\vx)&:=\tfrac1{4\pi}\int_\SS \vom I(\vx,\vom)\d\omega
=\vK^E_\nu(\vx) + {\mathcal \vK}[{\cal S}_\nu](\vx)\quad \text{ with}
\cr
\vK^E_\nu(\vx)&:=\tfrac1{4\pi}\int_\SS \vom I(\vx_\Sigma(\vx,\vom),\vom)\e^{-\int_0^{\tau_{\vx,\vom}}\kappa(\vx-\vom s)\d s}\d\omega
\cr&
 =
 \tfrac{1}{4\pi}\int_\Gamma (\vy-\vx)Q_\nu(\vy,\tfrac{\vy-\vx}{|\vy-\vx|})
 \frac{[(\vy-\vx)\cdot\vn(\vy)]_-}{|\vy-\vx|^4}\e^{-\int_{[\vx,\vy]}\kappa}\d\Gamma(\vy),
\cr
{\mathcal \vK}[{\cal S}_\nu](\vx)&:=
\tfrac1{4\pi}\int_\SS\int_0^{\tau_{\vx,\vom}} \vom\e^{-\int_0^s\kappa(\vx-\vom s')\d s'} {\cal S}_\nu(\vx-\vom s,\vom)\d s\d\omega
\\&
=
 \tfrac1{4\pi}\int_{\Omega} (\vy-\vx){\cal S}_\nu(\vy,\tfrac{\vy-\vx}{|\vy-\vx|})\frac{\e^{-\int_{[\vx,\vy]}\kappa}}{|\vy-\vx|^3}\d\vy.
 \end{aligned}
\end{equation}
\subsection{Iterative Method}
In \cite{JCP} it was shown that, in absence of $\vK_\nu$, i.e. $\beta_\nu=0$, the following is monotone and convergent. Its extension to $\beta_\nu>0$ is:
\subsubsection*{Algorithm}
\begin{enumerate}
\item Initialize  $J_\nu$, $\vK_\nu$ and $T$ (by zero, for instance).
\item Compute ${\cal S}_\nu$ by \eqref{ssource}.
\item Update $J_\nu$ by \eqref{gen1},\eqref{SE1}, \eqref{SE23} and $\vK_\nu$ by \eqref{gen1}.
\item Update $T$ by solving \eqref{RTb}.
\end{enumerate} 
\section{Extension to Reflective Conditions (RC)}
As in \cite{JCP2}\textcolor{red}{,} consider  boundary condition \eqref{RTc}.
Proposition \ref{prop2} can be extended to case with RC and \textcolor{red}{in} the case of non-multiple reflection, it becomes:
\begin{proposition}\label{prop3}
The same iterations are proposed with
\begin{equation}
J_\nu(\vx) = S^E_{\nu,1}(\vx) + S^E_{\nu,2}(\vx) + \bar{\mathcal J}[{\cal S}_\nu](\vx),
\quad
\vK_\nu(\vx) = \vK^E_{\nu,1}(\vx) + \vK^E_{\nu,2}(\vx) + \bar{\mathcal K}[{\cal S}_\nu](\vx),
\end{equation}
\begin{equation}
\begin{aligned}
& \label{SEE1}
S^E_{\nu,1}(\vx) =
 \tfrac{1}{4\pi}\int_\Gamma Q_\nu(\vy,\tfrac{\vy-\vx}{|\vy-\vx|})
 \frac{[(\vy-\vx)\cdot\vn(\vy)]_-}{|\vy-\vx|^3}\e^{-\int_{[\vx,\vy]}\kappa}\d\Gamma(\vy),
\\&  
S^E_{\nu,2}(\vx) =
 \sum_{n=1}^M\tfrac{1}{4\pi}\int_\Gamma  R_\nu(\vx'_n,\tfrac{\vx-\vx'_n}{|\vx-\vx'_n|})Q_\nu(\vy,\tfrac{\vx'_n-\vy}{|\vx'_n-\vy|})
 \cr &
 \hskip 3cm
 \times\frac{[(\vx'_n-\vy)\cdot\vn(\vy)]_- \e^{-\int_{[\vx,\vx'_n]\cup[\vx'_n,\vy]}\kappa}}{|\vx'_n-\vy|\;(|\vx-\vx'_n|+|\vx'_n-\vy|)^2}
 \d\Gamma(\vy).
 \\& 
\bar{\mathcal J}[{\cal S}_\nu](\vx) = 
 \tfrac1{4\pi}\int_{\Omega} \left[\frac{\e^{-\int_{[\vx,\vy]}\kappa}}{|\vy-\vx|^2}
 \right.
\cr &
 \hskip 2cm
\left.  +\sum_{n=1}^M \frac{\e^{-\int_{[\vx,\vx'_n]\cup[\vx'_n,\vy]}\kappa}}{(|\vx-\vx'_n|+|\vx'_n-\vy|)^2}
 R_\nu(\vx'_n,\tfrac{\vx-\vx'_n}{|\vx-\vx'_n|}) \right]{\cal S}_\nu(\vy,\tfrac{\vy-\vx}{|\vy-\vx|})\d\vy.
\end{aligned}
\end{equation}
\begin{equation}
\begin{aligned}& \label{REE1}
\vK^E_{\nu,1}(\vx) =
 \tfrac{1}{4\pi}\int_\Gamma (\vy-\vx )Q_\nu(\vy,\tfrac{\vy-\vx}{|\vy-\vx|})
 \frac{[(\vy-\vx)\cdot\vn(\vy)]_-}{|\vy-\vx|^3}\e^{-\int_{[\vx,\vy]}\kappa}\d\Gamma(\vy),
\\&  
\vK^E_{\nu,2}(\vx) =
 \sum_{n=1}^M\tfrac{1}{4\pi}\int_\Gamma  (\vx-\vx'_n)R_\nu(\vx'_n,\tfrac{\vx-\vx'_n}{|\vx-\vx'_n|})Q_\nu(\vy,\tfrac{\vx'_n-\vy}{|\vx'_n-\vy|})
 \cr &
 \hskip 3cm
\times \frac{[(\vx'_n-\vy)\cdot\vn(\vy)]_- \e^{-\int_{[\vx,\vx'_n]\cup[\vx'_n,\vy]}\kappa}}{|\vx'_n-\vy|\;(|\vx-\vx'_n|+|\vx'_n-\vy|)^2}
 \d\Gamma(\vy),
 \\&  
\bar{\mathcal \vK}[{\cal S}_\nu](\vx) = 
 \tfrac1{4\pi}\int_{\Omega} \left[\frac{\vy-\vx}{|\vy-\vx|^3}\e^{-\int_{[\vx,\vy]}\kappa}
\right.\\& \left.
\hskip1cm +\sum_{n=1}^M (\vx-\vx'_n)\frac{\e^{-\int_{[\vx,\vx'_n]\cup[\vx'_n,\vy]}\kappa}}{(|\vx-\vx'_n|+|\vx'_n-\vy|)^2}
 R_\nu(\vx'_n,\tfrac{\vx-\vx'_n}{|\vx-\vx'_n|}) \right]{\cal S}_\nu(\vy,\tfrac{\vy-\vx}{|\vy-\vx|})\d\vy.
\end{aligned}
\end{equation}
where $M$ is the number of RC boundaries and $\vx'_n$ is the point of reflection on the boundary of the ray going from $\vx$ to $\vy$ via $\vx'_n$.
\end{proposition}

All equations are discretized using a $P^1$ finite element framework.
Some integrals have a singular kernel, so a careful quadrature should be used. 

In \cite{JCP}, \cite{JCP2} a strategy is explained to compute all the integrals as a matrix vector product where the matrices are   hierarchical compressed ${\cal H}$-matrices~\cite{beb},\cite{beben},\cite{hack}.

In the grey case ($\kappa_\nu$ independent of $\nu$) we need 5+M matrices. This could be taxing in computer memory, but not in computing time because the core of the method is $N\ln N$ where $N$ is the number of finite element vertices.
\begin{figure}[htbp]
\begin{center}
\includegraphics[width=5cm]{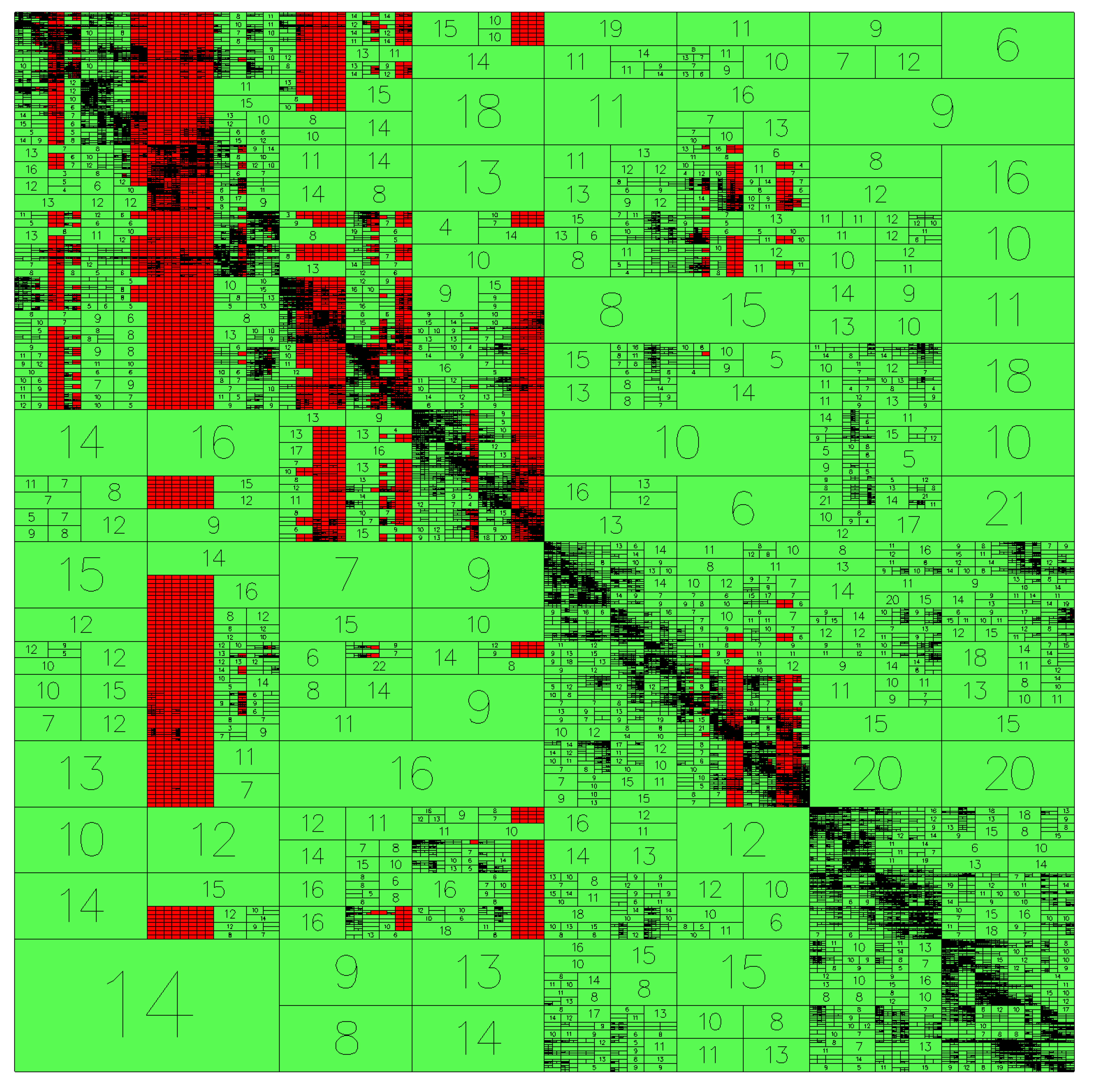}
\includegraphics[width=5cm]{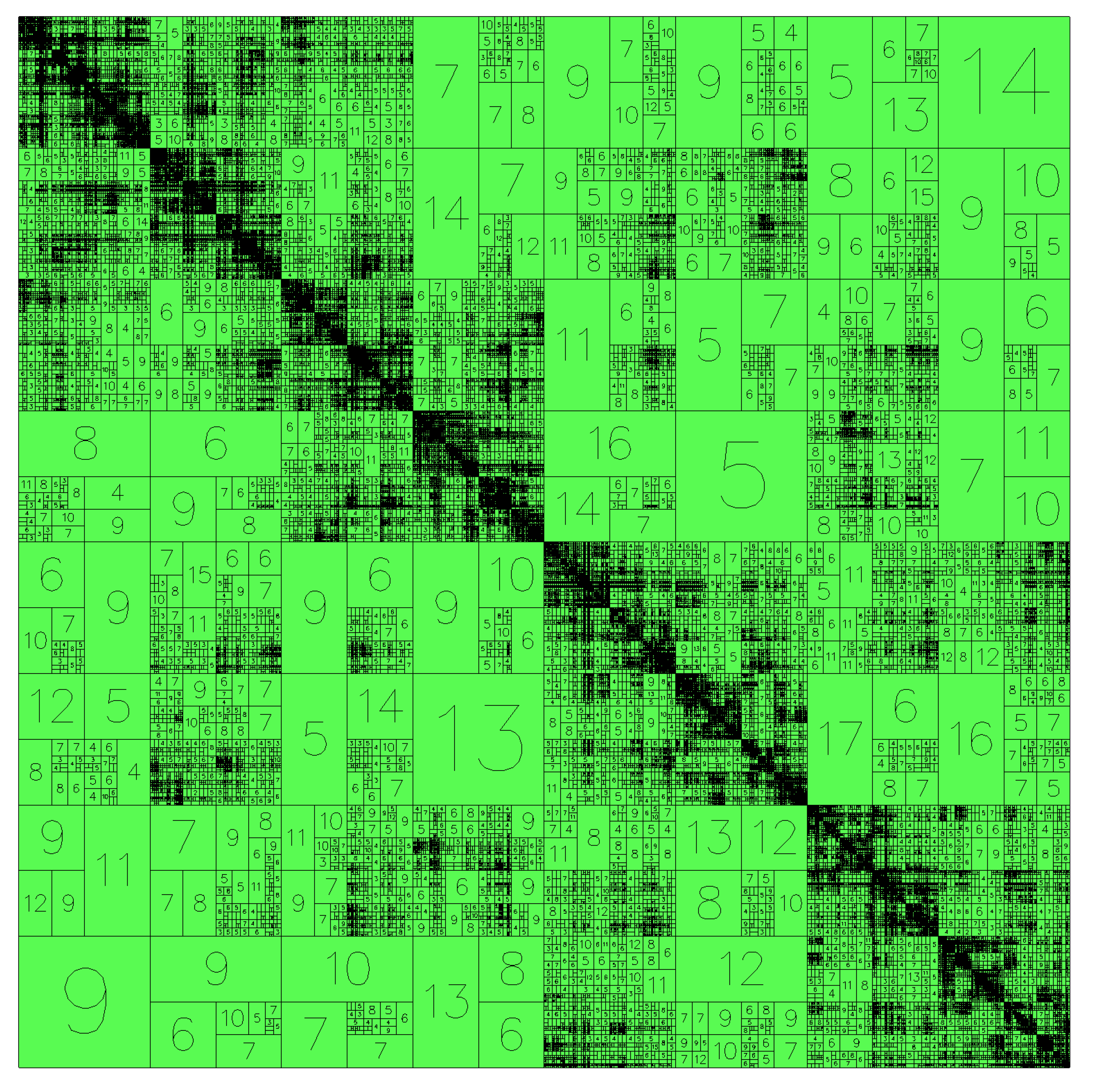}
\caption{Visual representations of the ${\cal H}$-matrices surface-to-volume on the left and volume-to-volume on the right, for the computation of the temperature in the Chamonix valley.
The level of compression is shown by light green-to-red and the numbers correspond to the rank of the truncated SVD approximation of each block.}
\label{matrix}
\end{center}
\end{figure}

In the non grey case we need P(5+M) matrices where $P$ is the number of values taken by the numerical approximation of $\kappa_\nu$.

\section{The Stratified One Dimensional Approximation}\label{sec3}

 Let us consider the Earth's atmosphere submitted to black body radiations from the sun at temperature $T_S$ and earth at $T_E$. The atmosphere is thin and the Sun is far. As a first approximation the ground is locally flat and has a negligible relief. Then all quantities are only function of the altitude $z$ and independent of $x$ and $y$. This, so called, stratified approximation of the RTE is  (see \cite{CBOP})
\begin{align}
{}&(\mu\partial_\tau +\kappa_\nu)I_\nu(\tau,\mu)=\kappa_\nu(1-a_\nu)B_\nu(T(\tau))
+ \tfrac12\kappa_\nu a_\nu\int_{-1}^1 p_\nu(\mu,\mu')I_\nu(\tau,\mu')\d\mu'\,,\label{RTSlab1}
\\
&I_\nu(0,\mu)=r_\nu I(0,-\mu)+\mu Q^+_\nu\,,\quad I_\nu(Z,-\mu)=\mu Q^-_\nu\,,\qquad 0<\mu<1\,,\label{RTSlab2}
\\
&\int_0^\infty\kappa_\nu(1-a_\nu)B_\nu(T(\tau))\d\nu=\int_0^\infty\kappa_\nu(1-a_\nu) J_\nu(\tau)\d\nu,\label{RTSlab3}
\end{align}
with    
\begin{equation}\label{jeq}
  Q^+_\nu=Q^EB_\nu(T_E),
  \quad
  Q^-_\nu=Q^SB_\nu(T_S),
  \quad
  J_\nu(\tau):=\tfrac12\int_{-1}^1I_\nu(\tau,\mu)\d\mu,
\,
\end{equation}
where $\mu$ is the cosine of $\vom$ with the vertical direction and $\tau$ is the optical thickness:
\[
\tau = \int_0^z\rho(z')\d z'.
\]   

\subsection{RTE with Non-Isotropic and Rayleigh Phase Function}\label{sec:5}

The scattering function is
$
p_\nu(\mu,\mu')=1+\beta_\nu\mu' 
$
in a cloud between altitude $Z_m$ and $Z_M$ and at high altitude $z>Z_M$ it is
the Rayleigh phase function 
$
p_\nu(\mu,\mu')=\tfrac3{16}(3-\mu^2)+\tfrac3{16}(3\mu^2-1)\mu'^2
$. It is combine\textcolor{red}{d} into one formula with two given altitude\textcolor{red}{s} and possibly frequency dependent functions $b_\nu(\tau)$ and $\beta_\nu(\tau)$ in $[0,1]$, 
\begin{equation}
p_\nu(\mu,\mu')=b_\nu+\beta_\nu\mu' + \tfrac3{16}(1-b_\nu)(3-\mu^2+(3\mu^2-1)\mu'^2).
\end{equation}
 Observe that $p_\nu(\mu,\mu')\ge0$ and $
\tfrac12\int_{-1}^1p_\nu(\mu,\mu')\d\mu'=1$ .
Keeping \eqref{RTSlab3} as the defining equation for $T$, given $I$, the problem becomes
\begin{equation}\label{RTSlab2R}
\left\{
\begin{aligned}
&(\mu\partial_\tau +\kappa_\nu)I_\nu(\tau,\mu)={\cal S}_\nu(\tau)
:=   {\cal R}_\nu(\tau) + {\cal P}_\nu(\tau)\mu^2\,,
\\
&I_\nu(0,\mu)=r_\nu I(0,-\mu)+\mu Q^+_\nu\,,\quad I_\nu(Z,-\mu)=\mu Q^-_\nu\,,\qquad 0<\mu<1\,,
\end{aligned}
\right.
\end{equation}
with 
\begin{equation}
\label{DefPRS}
\begin{aligned}
 K_\nu(\tau):=&\tfrac12\int_{-1}^1\mu I_\nu(\tau,\mu)\d\mu\, ,
 \qquad L_\nu(\tau)=\tfrac12\int_{-1}^1\mu^2I_\nu\d\mu\, ,
\\
{\cal R}_\nu(\tau) :=& \kappa_\nu a_\nu\left(  (\tfrac7{16}b+\tfrac9{16}) J_\nu(\tau)+\beta{K_\nu}(\tau)- \tfrac3{16}(1-b)L_\nu(\tau) \right) +\kappa_\nu(1-a_\nu)B_\nu(T(\tau))\,,
\\
{\cal P}_\nu(\tau):=& \tfrac3{16}(1-b)(-J_\nu(\tau)+3{L_\nu}(\tau))
\end{aligned}
\end{equation}
and \eqref{RTSlab3}. 

\subsubsection*{Algorithm}
\begin{enumerate}
\item
Initialize $J^0_\nu(\tau,\mu)=0$, $K^0_\nu(\tau,\mu)=0$, $L^0_\nu(\tau,\mu)=0$ and $T^0(\tau)=0$.  
\item Compute $I^{n+1}$ by \eqref{RTSlab2R} with ${\cal S}_\nu^n$.
\item Update $J^{n+1}_\nu(\tau,\mu)$, $K^{n+1}_\nu(\tau,\mu)$, $L^{n+1}_\nu(\tau,\mu)$ by \eqref{DefPRS}, then  ${\cal R}_\nu(\tau)$ and ${\cal P}_\nu(\tau)$.
\item Then, compute $T^{n+1}$ by solving \eqref{RTSlab3}.
\end{enumerate}

\subsection{Implementation}
Recall the definition of the exponential integrals 
\begin{equation}\label{expintegr}
E_p(X):=\int_0^1e^{-X/\mu}\mu^{p-2}\d\mu\,,\qquad X>0\,.
\end{equation}
For $i=3,4,5$, let%
\begin{equation}\label{source}
\begin{aligned}
S_i(\nu,\tau) &:= \tfrac12 E_i(\kappa_\nu\tau) Q^+(\tau) + \tfrac{(-1)^{i-1}}2( E_{i}(\kappa_\nu(Z-\tau))+ r_\nu E_{i}(\kappa_\nu(Z+\tau)))Q^-(\tau),
\\ 
F_i(\tau,t) &:= \tfrac12 E_3(\kappa_\nu\vert \tau-t\vert )+\tfrac{r_\nu}2 E_3(\kappa_\nu\vert \tau+t\vert ).
\end{aligned}
\end{equation}
\begin{proposition}

With \eqref{DefPRS} and \eqref{source} the quantities needed by the algorithm above are given by
\begin{equation}
\label{algo:two}
\begin{aligned}
 J^{n+1}_\nu(\tau)=&S_3(\nu,\tau) +\tfrac12\int_0^Z F_1(\tau,t ){\cal R}_\nu^n(\tau)\d t
 +\tfrac12\int_0^Z F_3(\tau,t ){\cal P}_\nu^n(\tau)\d t\,,
\\
{K^{n+1}_\nu}(\tau)=&S_4(\nu,\tau)
+\tfrac12\int_0^\tau F_2(\tau,t ){\cal R}_\nu^n(t)\d t
+\tfrac12\int_0^\tau F_4(\tau,t ){\cal P}_\nu^n(t)\d t
\\&
-\tfrac12\int_\tau^Z F_2(\tau,t ){\cal R}_\nu^n(t)\d t
-\tfrac12\int_\tau^Z F_4(\tau,t ){\cal P}_\nu^n(t)\d t\,,
\\
{L^{n+1}_\nu}(\tau)=&S_5(\nu,\tau)
+\tfrac12\int_0^Z F_3(\tau,t ){\cal R}_\nu^n(\tau)\d t
+\tfrac12\int_0^Z F_5(\tau,t ){\cal P}_\nu^n(\tau)\d t\,.
\end{aligned}
\end{equation}
\end{proposition}
\begin{proof}~~~
For clarity let us assume that $r_\nu=0$. For the general case see \cite{FGOP}.
Applying the method of characteristics shows that
\begin{equation}\label{IntForm}
\begin{aligned}
I_\nu&(\tau,\mu)=\mu e^{-\frac{\kappa_\nu\tau}{\mu}}Q^+_\nu\mathbf 1_{\mu>0}+ |\mu|e^{-\frac{\kappa_\nu(Z-\tau)}{\vert\mu\vert}}Q^-_\nu\mathbf 1_{\mu<0}
\\
&+\mathbf 1_{\mu>0}\int_0^\tau e^{-\frac{\kappa_\nu(\tau-t)}{\mu}}\tfrac{\kappa_\nu}{\mu}{\cal S_\nu}(t)\d t
+\mathbf 1_{\mu<0}\int_\tau^Ze^{-\frac{\kappa_\nu(t-\tau)}{\vert \mu\vert }}\tfrac{\kappa_\nu}{|\mu|}{\cal S_\nu}(t)\d t\,.
\end{aligned}
\end{equation}
Therefore with $k=0$ or $k=2$,
\begin{equation}
\begin{aligned}&
\int_{-1}^1\mu^k I_\nu(\tau,\mu)\d\mu=\int_{-1}^1|\mu|^{k+1}\Big(e^{-\frac{\kappa_\nu\tau}{\mu}}Q^+_\nu\mathbf 1_{\mu>0}+e^{-\frac{\kappa_\nu(Z-\tau)}{\vert\mu\vert}}Q^-_\nu\mathbf 1_{\mu<0}\Big)\d\mu
\\
&+\int_{-1}^1|\mu|^{k-1}\Big(\mathbf 1_{\mu>0}\int_0^\tau e^{-\frac{\kappa_\nu(\tau-t)}{\mu}}\kappa_\nu{\cal S_\nu}(t)\d t
+\mathbf 1_{\mu<0}\int_\tau^Ze^{-\frac{\kappa_\nu(t-\tau)}{\vert \mu\vert }}\kappa_\nu{\cal S_\nu}(t)\d t\Big)\d\mu\,
\\&
=\int_{0}^1|\mu|^{k+1}e^{-\frac{\kappa_\nu\tau}{\mu}}Q^+_\nu\d\mu +
\int_{-1}^0|\mu|^{k+1}e^{-\frac{\kappa_\nu(Z-\tau)}{\vert\mu\vert}}Q^-_\nu\d\mu
\\
&+\int_{0}^1|\mu|^{k-1}\int_0^\tau e^{-\frac{\kappa_\nu(\tau-t)}{\mu}}\kappa_\nu{\cal S_\nu}(t)\d t
+\int_{-1}^0|\mu|^{k-1}\int_\tau^Ze^{-\frac{\kappa_\nu(t-\tau)}{\vert \mu\vert }}\kappa_\nu{\cal S_\nu}(t)\d t\d\mu
\\&
=\int_{0}^1|\mu|^{k+1}\left(e^{-\frac{\kappa_\nu\tau}{\mu}}Q^+_\nu
+ e^{-\frac{\kappa_\nu(Z-\tau)}{\vert\mu\vert}}Q^-_\nu\right)\d\mu
\\
&+\int_0^\tau \int_{0}^1|\mu|^{k-1}e^{-\frac{\kappa_\nu(\tau-t)}{\mu}}\kappa_\nu
({\cal R}_\nu(t) + {\cal P}_\nu(t)\mu^2)\d\mu\d t
\\&
+\int_\tau^Z\int_{0}^1|\mu|^{k-1}e^{-\frac{\kappa_\nu(t-\tau)}{\vert \mu\vert }}\kappa_\nu
({\cal R}_\nu(t) + {\cal P}_\nu(t)\mu^2)\d\mu\d t
\\&
= E_{k+3}(\kappa_\nu\tau)Q^+_\nu + E_{k+3}(\kappa_\nu(Z-\tau))Q^- 
\\&
+ \int_0^\tau E_{k+1}(\kappa_\nu(\tau-t))R_\nu(t)\d t  
+ \int_\tau^Z E_{k+1}(\kappa_\nu(t-\tau))R_\nu(t)\d t
\\&
+ \int_0^\tau E_{k+3}(\kappa_\nu(\tau-t))P_\nu(t)\d t 
+ \int_\tau^Z E_{k+3}(\kappa_\nu(t-\tau))P_\nu(t)\d t.
\end{aligned}
\end{equation}
For $k=1$, the same computation gives
\begin{equation}
\begin{aligned}&
\int_{-1}^1\mu I_\nu(\tau,\mu)\d\mu=\int_{-1}^1\mu^2\Big(e^{-\frac{\kappa_\nu\tau}{\mu}}Q^+_\nu\mathbf 1_{\mu>0}-e^{-\frac{\kappa_\nu(Z-\tau)}{\vert\mu\vert}}Q^-_\nu\mathbf 1_{\mu<0}\Big)\d\mu
\\
&+\int_{-1}^1\Big(\mathbf 1_{\mu>0}\int_0^\tau e^{-\frac{\kappa_\nu(\tau-t)}{\mu}}\kappa_\nu{\cal S_\nu}(t)\d t
-\mathbf 1_{\mu<0}\int_\tau^Ze^{-\frac{\kappa_\nu(t-\tau)}{\vert \mu\vert }}\kappa_\nu{\cal S_\nu}(t)\d t\Big)\d\mu\,
\\&
=\int_{0}^1\mu^{2}e^{-\frac{\kappa_\nu\tau}{\mu}}Q^+_\nu\d\mu -
\int_{-1}^0\mu^{2}e^{-\frac{\kappa_\nu(Z-\tau)}{\vert\mu\vert}}Q^-_\nu\d\mu
\\
&+\int_{0}^1\int_0^\tau e^{-\frac{\kappa_\nu(\tau-t)}{\mu}}\kappa_\nu{\cal S_\nu}(t)\d t
-\int_{-1}^0\int_\tau^Ze^{-\frac{\kappa_\nu(t-\tau)}{\vert \mu\vert }}\kappa_\nu{\cal S_\nu}(t)\d t\d\mu
\end{aligned}
\end{equation}
\begin{equation}
\begin{aligned}&
=\int_{0}^1\mu^{2}\left(e^{-\frac{\kappa_\nu\tau}{\mu}}Q^+_\nu- e^{-\frac{\kappa_\nu(Z-\tau)}{\vert\mu\vert}}Q^-_\nu\right)\d\mu
\\
&+\int_0^\tau \int_{0}^1 e^{-\frac{\kappa_\nu(\tau-t)}{\mu}}\kappa_\nu
({\cal R}_\nu(t) + {\cal P}_\nu(t)\mu^2)\d\mu\d t
\\&
-\int_\tau^Z\int_{0}^1e^{-\frac{\kappa_\nu(t-\tau)}{\vert \mu\vert }}\kappa_\nu
({\cal R}_\nu(t) + {\cal P}_\nu(t)\mu^2)\d\mu\d t
\\&
= E_{4}(\kappa_\nu\tau)Q^+_\nu - E_{4}(\kappa_\nu(Z-\tau))Q^- 
\\&
+ \int_0^\tau E_{2}(\kappa_\nu(\tau-t)){\cal R}_\nu(t)\d t  
- \int_\tau^Z E_{2}(\kappa_\nu(t-\tau)){\cal R}_\nu(t)\d t
\\&
+ \int_0^\tau E_{4}(\kappa_\nu(\tau-t)){\cal P}_\nu(t)\d t 
- \int_\tau^Z E_{4}(\kappa_\nu(t-\tau)){\cal P}_\nu(t)\d t.
\end{aligned}
\end{equation}
\end{proof}

\subsection{Numerical Simulations}
At the top of the atmosphere, we are told, the solar radiation intensity is $340W/m^2$ but the atmosphere reflects 30\% of it and our RTE does not handle volumic reflection. At 10 or 12km the sunlight is diffused, so we apply an intensity proportional to $\vom_3$. Paris being at latitude $45^0$, the radiation is divided by $\sqrt{2}$ . Hence an effective number is $Q_{sun}=162W/m^2$. And yet a good portion is spent for the evaporation of water and into convection. Arbitrarily we keep only half $80W/m^2$.

Recall the scalings defined in \cite{FGOP3}.
The frequencies and the temperatures are scaled as follows (primes label scaled variables),
\[
\nu'=10^{-14}\nu,~~T'=10^{-14}\frac k{\hh}T=10^{-14}\frac{1.381\cdot 10^{-23}}{6.626\cdot 10^{-34}}T = \frac{T}{4798}.
\] 
Planck's function is written as
\[
B_0\frac{\nu'^3}{e^\frac{\nu'}{T'}-1}, ~~\hbox{ with } B_0=\frac{2\hh}{c^2}10^{42}=\frac{2\times 6.626\cdot 10^{-34}}{2.998^2\cdot 10^{16}}10^{42}= 1.4744\cdot 10^{-8}.
\]
Consequently the computed  intensity, $I'_\nu$, is the physical intensity divided by $B_0$.
The scaled Sun temperature is $T'_S=5800/4798=1.209$.  According to \eqref{jeq},
\[
\int_0^\infty Q^+_\nu=80=Q^S\int_0^\infty B_0\frac{\nu'^3}{e^\frac{\nu'}{T'_S}-1}(10^{14}\d\nu')
= Q^S 1.4744\cdot 10^{6}\frac{\pi^4}{15}1.209^4.
\]
This leads to $Q^S=4 \cdot 10^{-6}$.

Similarly Earth is at $T'_E=288/4798=0.06$ and emits about $350W/m^2$, but some is used for the evaporation of water so we have kept only  $300W/m^2$:
\[
300 = Q^E \int_0^\infty B_0\frac{\nu'^3}{e^\frac{\nu'}{T'_E}-1}(10^{14}\d\nu')
= Q^E 1.4744\cdot 10^{6}\frac{\pi^4}{15}0.06^4= Q^E 124.3.
\]
This leads to $Q^E=2.41 $.

 The air density is $1-\tfrac34 z$. The earth reflective albedo is set at $r_\nu=0.3$.  Isotropic scattering is set at $a_\nu=0.3$.

$\kappa_\nu$ is given by the Gemini experimental program.  
To measure the temperature perturbation due to an increase of \texttt{CO}$_2$ (resp \texttt{NO}$_x$)in the atmosphere, we set $\kappa_\nu=1$ in the range $(3/16,3/14)$ (resp. $(3/7,3/5)$.

\subsubsection{Isotropic Scattering}
Two tests were performed with the stratified 1D approximation, isotropic scattering, no cloud and Gemini $\kappa_\nu$ + Rayleigh scattering equal to $0.3$ for $\nu>3$ and $\tau>0.7$.
One computation is with an infrared source $Q^E$ at $\tau=0$ only (similar to the situation at night) and another with a sunlight source $Q^S$ at $\tau=Z$ only.

Results  are on Figures \ref{onedE}, \ref{onedS}. Computations are done with 3 different atmospheres.  One is conform to the Gemini measurements, the other two imitate an addition of \texttt{CO}$_2$ or methane.

One can observe a small additional greenhouse effect due to addition of \texttt{CO}$_2$ but a cooling effect at high altitude.  The opposite is observed for addition of \texttt{NO}$_x$.

\begin{figure}[h]
\begin{minipage}[b]{0.5\textwidth} 
\begin{center}
\begin{tikzpicture}[scale=0.75]
\begin{axis}[legend style={at={(1,1)},anchor=north east}, compat=1.3,
   xlabel= {Altitude},
  ylabel= {T}
  ]
\addplot[thick,solid,color=red,mark=none, mark size=1pt] table [x index=0, y index=1]{figE/temperature00.txt};
\addlegendentry{no GHG}
\addplot[thick,solid,color=blue,mark=none, mark size=1pt] table [x index=0, y index=1]{figE/temperature01.txt};
\addlegendentry{with \texttt{CO}$_2$}
\addplot[thick,solid,color=green,mark=none, mark size=1pt] table [x index=0, y index=1]{figE/temperature02.txt};
\addlegendentry{with \texttt{NO}$_x$}
\end{axis}
\end{tikzpicture}
\caption{\label{onedE} Temperature T versus altitude z for 3 values of $\nu\mapsto\kappa_\nu$.
Only the earth radiation is taken as a source. (night time).
}
\end{center}
\end{minipage}
\hskip0.2cm
\begin{minipage}[b]{0.45\textwidth} 
\begin{center}
\hskip-0.4cm
\begin{tikzpicture}[scale=0.75]
\begin{axis}[legend style={at={(1,1)},anchor=north east}, compat=1.3,
   ymax=-10,
   xlabel= {Altitude},
  ylabel= {T}
  ]
\addplot[thick,solid,color=red,mark=none, mark size=1pt] table [x index=0, y index=1]{figS/temperature00.txt};
\addlegendentry{no GHG}
\addplot[thick,solid,color=blue,mark=none, mark size=1pt] table [x index=0, y index=1]{figS/temperature01.txt};
\addlegendentry{with \texttt{CO}$_2$}
\addplot[thick,solid,color=green,mark=none, mark size=1pt] table [x index=0, y index=1]{figS/temperature02.txt};
\addlegendentry{with \texttt{NO}$_x$}
\end{axis}
\end{tikzpicture}
\caption{\label{onedS}  T vs $z$ for 3 values of $\nu\mapsto\kappa_\nu$.
Only  the Sun radiation is taken as a source.}
\end{center}
\end{minipage}
\end{figure}

\subsubsection{Anisotropic Scattering in a Cloud}
Next, still with the stratified approximation, a cloud is added between altitude $Z_m=7000$m and $Z_M=9000$m with anisotropic scattering coefficient $\beta=\tfrac{a_\nu}2$, $b=1,~ a_\nu=0.3\tfrac4{(Z_M-Z_m)^2}(z-Z_m)^+(Z_M-z)^+$.  
 Above $Z_M$ Rayleigh scattering $0.3 1_{\nu>3}$, replaces the anisotropic scaattering.
 
 Results are shown on Figure \ref{oned2}, \ref{oned4} and \ref{oned3}.
The strong effect of the cloud is shown on Figure \textcolor{red}{\ref{oned4}}. Figure \ref{oned3} shows how important it is to take all the details of $\kappa_\nu$ into account.

\begin{figure}[h]
\begin{minipage}[b]{0.5\textwidth} 
\begin{center}
\begin{tikzpicture}[scale=0.75]
\begin{axis}[legend style={at={(1,1)},anchor=north east}, compat=1.3,
   xlabel= {Altitude},
  ylabel= {T}
  ]
\addplot[thick,solid,color=red,mark=none, mark size=1pt] table [x index=0, y index=1]{fig2/temperature00.txt};
\addlegendentry{no GHG}
\addplot[thick,solid,color=blue,mark=none, mark size=1pt] table [x index=0, y index=1]{fig2/temperature01.txt};
\addlegendentry{with \texttt{CO}$_2$}
\addplot[thick,solid,color=green,mark=none, mark size=1pt] table [x index=0, y index=1]{fig2/temperature02.txt};
\addlegendentry{with \texttt{NO}$_x$}
\end{axis}
\end{tikzpicture}
\caption{\label{oned2} $T$ versus $z$ computed with a cloud}
\end{center}
\end{minipage}
\hskip0.2cm
\begin{minipage}[b]{0.5\textwidth} 
\begin{center}
\begin{tikzpicture}[scale=0.75]
\begin{axis}[legend style={at={(1,1)},anchor=north east}, compat=1.3,
   ymax=2,
   xlabel= {altitude},
  ylabel= {T}
  ]
\addplot[thick,solid,color=red,mark=none, mark size=1pt] table [x index=0, y index=1]{fig1/classeur.txt};
\addlegendentry{no GHG}
\addplot[thick,solid,color=blue,mark=none, mark size=1pt] table [x index=0, y index=2]{fig1/classeur.txt};
\addlegendentry{with \texttt{CO}$_2$}
\addplot[thick,solid,color=green,mark=none, mark size=1pt] table [x index=0, y index=3]{fig1/classeur.txt};
\addlegendentry{with \texttt{NO}$_x$}
\end{axis}
\end{tikzpicture}
\caption{\label{oned4} Difference T$_{nocloud}$  - T$_{cloud}$}
\end{center}
\end{minipage}
\end{figure}

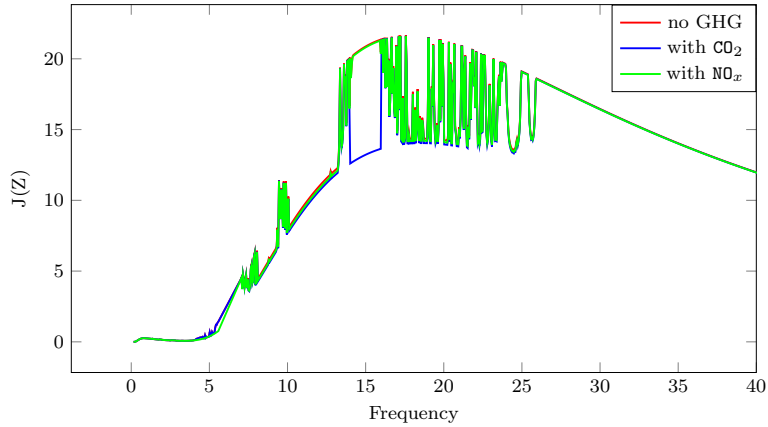
\begin{figure}[h]
\begin{center}
\begin{tikzpicture}[scale=0.9]
\begin{axis}[legend style={at={(1,1)},anchor=north east}, compat=1.3,
   xmax=40,
   xlabel= {Frequency},
  ylabel= {J(Z)},
  width=1.\textwidth,
        height=0.6\textwidth,
  ]
\addplot[thick,solid,color=red,mark=none, mark size=1pt] table [x index=0, y index=1]{fig1/imean00Z.txt};
\addlegendentry{no GHG}
\addplot[thick,solid,color=blue,mark=none, mark size=1pt] table [x index=0, y index=1]{fig1/imean01Z.txt};
\addlegendentry{with \texttt{CO}$_2$}
\addplot[thick,solid,color=green,mark=none, mark size=1pt] table [x index=0, y index=1]{fig1/imean02Z.txt};
\addlegendentry{with \texttt{NO}$_x$}
\end{axis}
\end{tikzpicture}
\caption{\label{oned3} Radiative intensity $J_\nu(Z)$ versus $\nu$ for the 3 $\nu\mapsto\kappa_\nu$ cases.}
\end{center}
\end{figure}

Finally, in Table \ref{radiative} total radiative intensities $\int_0^\infty J_\nu(\tau)$ are shown at $\tau=0$ and $\tau=Z$ in a variety of situation\textcolor{red}{s} to see the {effect} of clouds, scattering, earth albedo.

\begin{table}[htp]
\caption{Radiative intensities for the Earth radiations + Sunlight (first 2 {rows}) and the Sunlight only (last 2 {rows}) when all terms are present and when one of the terms is removed.}
\begin{center}
\begin{tabular}{|c|c|c|c|c|c|}
\hline
& everything & no cloud & no scattering & no earth albedo & no absorption\cr
\hline
$\tau=0$ & 9.52495 & 9.39105 & 9.52441 & 7.34548 & 7.21413\cr
$\tau=Z$ & 5.98738 & 5.99558 & 5.98622 & 4.93935 & 7.14319\cr
\hline
$\tau=0$ & 1.44484 & 1.43589 & 1.4431 & 1.51029 & 1.51029\cr
$\tau=Z$ & 1.91376 & 1.90286 & 1.91093 & 1.73976 & 1.51995\cr
\hline
\end{tabular}
\end{center}
\label{radiative}
\end{table}%

In the presence of \texttt{CO}$_2$ the total radiative energy at ground level (resp. $\tau=Z$), is $9.91316$ (resp. $5.72292$).  With \texttt{NO}$_x$ it is 9.55564  (resp. 5.8434).  These numbers must be compared with the first number on the left in the first (resp 2$^{nd}$) {row} in Table \ref{radiative}.

\begin{figure}[htbp]
\begin{center}
\begin{tikzpicture}[scale=0.75]
\begin{axis}[legend style={at={(1,1)},anchor=north east}, compat=1.3,
   xlabel= {Altitude},
  ylabel= {T}
  ]
\addplot[thick,dotted,color=red,mark=none, mark size=1pt] table [x index=0, y index=1]{fig1/temperature000.txt};
\addlegendentry{$\beta=0$}
\addplot[thick,solid,color=blue,mark=none, mark size=1pt] table [x index=0, y index=1]{fig1/temperature05.txt};
\addlegendentry{$\beta=0.50$}
\addplot[thick,dashed,color=green,mark=none, mark size=1pt] table [x index=0, y index=1]{fig1/temperature07.txt};
\addlegendentry{ $\beta=0.75$}
\end{axis}
\end{tikzpicture}
\caption{\label{oned7} $T$ versus $z$ computed with a cloud with anisotropic scattering of coefficient $\beta$.}
\end{center}
\end{figure}
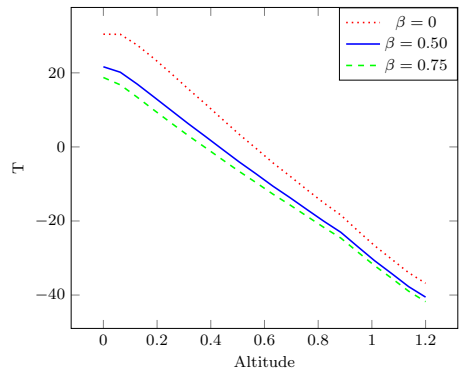

\section{A 3D formulation with a Stratified Part}
To implement in 3D the Earth albedo (reflective condition by the ground) in all generality is quite complicated. To simplify the computation we consider the reflection of the sunlight only and we use the linearity, with respect to the sources, of the equation for $I_\nu$, for a fixed temperature field.

For a portion $\Omega$ of atmosphere between the ground $\vg=(\vx_1,\vx_2,\vx_3=g(\vx_1,\vx_2))$, $g\ge 0$, and the lower stratosphere $\vx_3=Z$, consider
\begin{equation}
\begin{aligned}
\label{oneaz}&
\vom\cdot\nablav_\vx {I}_\nu+\rho\bar\kappa_\nu a_\nu\left[{I}_\nu-{\frac{1}{4\pi}\int_{\S^2}} p_\nu(\omegav,\omegav'){I}_\nu(\omegav'){\d}\omega'\right]
\\&
\hskip6cm = \rho\bar\kappa_\nu(1-a_\nu) [B_\nu(T)-{I}_\nu],
\\  &
I_\nu|_{\vg, \vom\cdot\vn<0} = Q^E (\vom\cdot\vn)^- B_\nu(T_E), \quad
I_\nu|_{Z,\vom_3<0} = Q^S \vom_3^- B_\nu(T_S).
\end{aligned}
\end{equation}
The solution is $I_\nu=I_\nu^E+I_\nu^S$ with $I_\nu^E$ computed with the condition $I_\nu|_{\vg}$ only and the source term in $T$, and $I_\nu^S$ computed with $I_\nu|_{Z}$ only and $T=0$:
\begin{equation}
\begin{aligned}
\label{onea4}&
\left\{\begin{matrix}\omegav\cdot\nablav_\vx I_\nu^S+\rho\bar\kappa_\nu a_\nu\left[I_\nu^S-{\frac{1}{4\pi}\int_{\S^2}} p_\nu(\omegav,\omegav')I_\nu^S(\omegav'){\d}\omega'\right]
+  \rho\bar\kappa_\nu(1-a_\nu) I_\nu^S = 0,
\cr
I^S_\nu|_{Z,\vom_3<0} = Q^S \vom_3^- B_\nu(T_S),\qquad I^S_\nu|_{0,\vom_3>0} = 0,~~~~~~~~~~~~~~~~~
\end{matrix}\right.
\\ &
\left\{\begin{matrix} \omegav\cdot\nablav_\vx I_\nu^E+\rho\bar\kappa_\nu a_\nu\left[I_\nu^E-{\frac{1}{4\pi}\int_{\S^2}} p_\nu(\omegav,\omegav')I_\nu^E(\omegav'){\d}\omega'\right]~~~~~~~~~~~~~~~~~~~
\cr
\hskip6cm = \rho\bar\kappa_\nu(1-a_\nu) [B_\nu(T)-I_\nu^E],
\cr
I^E_\nu|_{g, \vom\cdot\vn<0} = Q^E (\vom\cdot\vn)^- B_\nu(T_E) - I^S_\nu|_{\vg,\vom}, \quad I^E_\nu|_{Z,\vom_3<0}  =0.
\end{matrix}\right.
\end{aligned}
\end{equation}

For $I_\nu^S$ the stratified approximation is justified. Furthermore, it can be precomputed beforehand.  

To account for the fact that a fraction $r_\nu(\vx)$ of the sunlight is reflected by the Earth surface, we consider a partial reflective  condition  
\[
I^E_\nu|_{\vg,\vom\cdot\vn<0} = -I^S_\nu|_{\vg,\vom} + Q^E (\vom\cdot\vn)^- B_\nu(T_E) + r_\nu(\vg) I_\nu^S|_{\vg,\vom-2(\vom\cdot\vn)\vn}\,,
\]
but $I^S_\nu|_{\vg,\vom}\approx 0$ when $\vom\cdot\vn<0$, so, neglecting scattering, $a_\nu=0$ (see also \cite{pekka})
\[
I_\nu^S(\vx,\vom) = \vom_3^- Q^S B_\nu(T_S)\e^{-\kappa_\nu\frac{Z-\tau(\vx_3)}{\vom_3^-}}.
\]
Therefore, with $\vom\cdot\vn<0$ and $\vx$ on the ground,
\begin{equation}
\begin{aligned}
-I^S_\nu|_{\vg,\vom} &+ r_\nu(\vx) I_\nu^S|_{\vg,\vom-2(\vom\cdot\vn)\vn}
\\ &
= (\vom_3-2(\vom\cdot\vn)\vn_3)^- Q^SB_\nu(T_S) r_\nu(\vg)\e^{-\frac{\kappa_\nu (Z-\tau(g))}{(\vom_3-2(\vom\cdot\vn)\vn_3)^-}}.
\end{aligned}
\end{equation}
For a flat ground, it is equal to $\vom_3^+Q^SB_\nu(T_S) r_\nu(\vg)\e^{-\frac{\kappa_\nu Z}{\vom_3^+}}$ and its angle average on $\SS$ is $\frac12 Q^SB_\nu(T_S) r_\nu(\vg) E_3(\kappa_\nu Z)$ so that it makes sense to solve the problem with the following conditions,
\begin{equation}
\begin{aligned}
I_\nu|_{\vg,\vom\cdot\vn<0} &=   (\vom\cdot\vn)^- \left(Q^E B_\nu(T_E) +   Q^S B_\nu(T_S) 2 r_\nu(\vg) E_3(\kappa_\nu Z)\right), 
\\
 I_\nu|_{Z,\vom_3<0} & =
(\vom\cdot\vn)^- Q^S  B_\nu(T_S) .
\end{aligned}
\end{equation}
The valley of Chamonix is considered in the same physical conditions as above, i.e. with a cloud between 7000m and 9000m. However we have not implemented the anisotropic scattering yet, so $\beta=0$.

All 4 vertical boundaries which limit the domain $\Omega$ are reflective.
Results are shown on Figures \ref{chamT} and \ref{fig3D}.
\begin{figure}[h]
\begin{center}
\begin{tikzpicture}[scale=0.75]
\begin{axis}[legend style={at={(1,1)},anchor=north east}, compat=1.3,
   xmin=0.09,
   xlabel= {Altitude},
  ylabel= {T}
  ]
\addplot[thick,solid,color=red,mark=none, mark size=1pt] table [x index=0, y index=1]{fig1/noonScloudScatttempe.txt};
\addlegendentry{grey case $\kappa_\nu=0.45$}
\end{axis}
\end{tikzpicture}
\caption{\label{chamT} Temperature above the city of Chamonix versus altitude in a grey case. The growth of $T$ from the ground up is partially due to quadrature errors, partially to the parabolic effect of mountains.}
\end{center}
\end{figure}
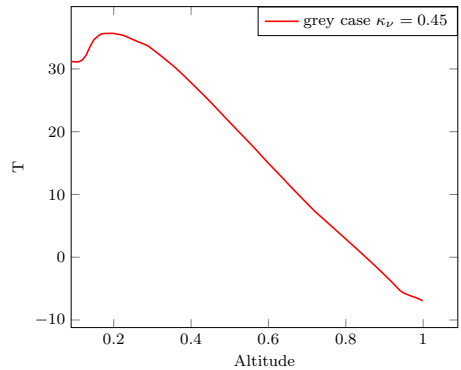

\begin{figure}[htbp]
\begin{center}
\includegraphics[width=13cm]{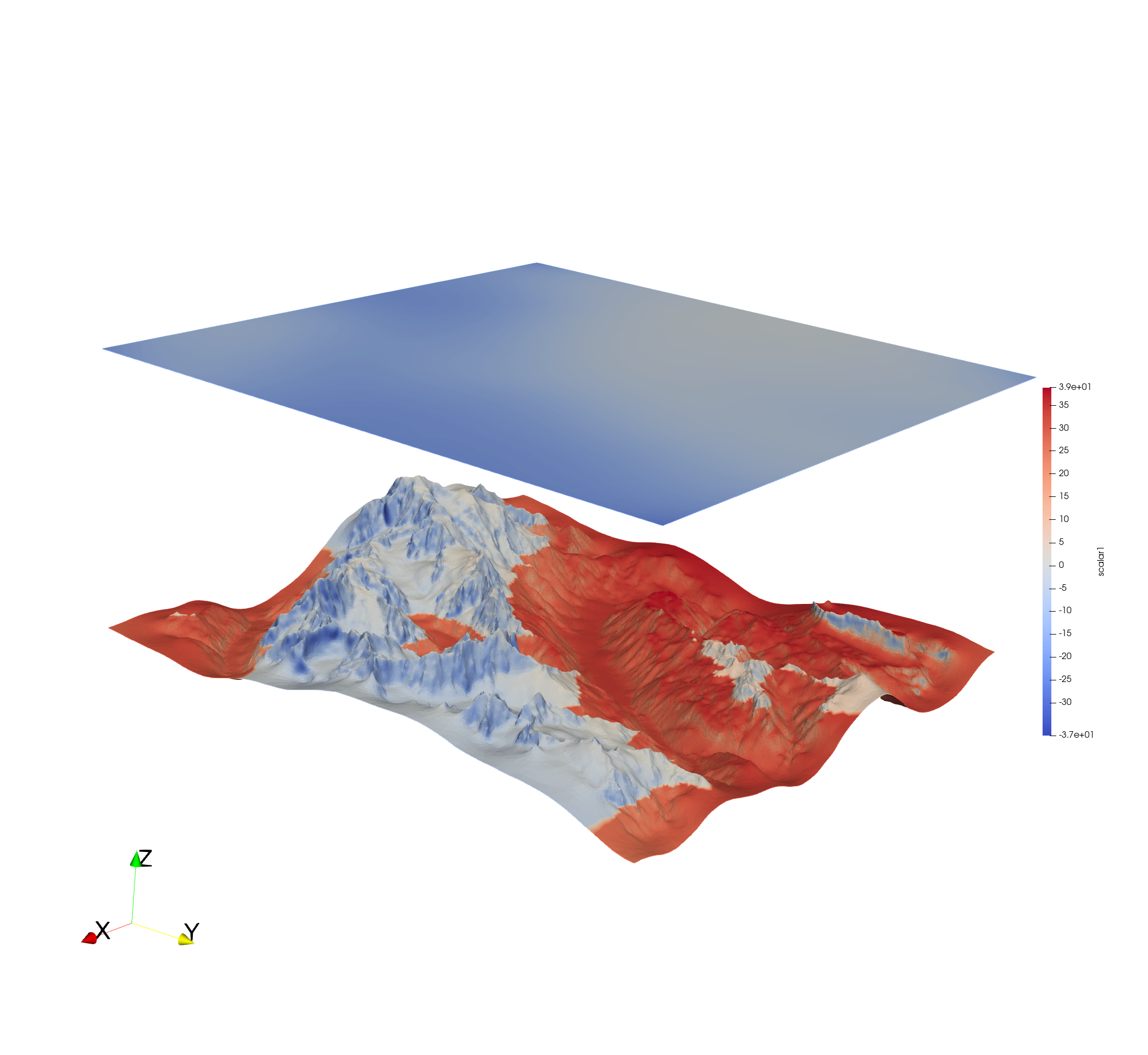}

\caption{Ground and high altitude temperature in an atmosphere receiving sunlight and Earth infrared radiation. In addition 0.3\% of the sunlight is reflected in the normal direction of the ground and 0.7\% if the ground is covered by snow. }
\label{fig3D}
\end{center}
\end{figure}

\section*{Conclusion}

The integro-differential formulation of the RTE and its solution by iterations on the source has been extended here to handle anisotropic scattering.

The iterative part of the method is $O(N\ln N)$, thanks to an efficient use of ${\cal H}$-matrices.

The precision is good enough to evaluate the effect of sensitive parameters for the study of contrails.  Most of the time the stratified 1D approximation should suffice, but in complex cases with high relief the 3D formulation is needed.

\bibliographystyle{siamplain}
\bibliography{references}

\end{document}